\documentclass[12pt]{amsart}
\usepackage{amssymb,amsmath}

\headsep=1cm \numberwithin{equation}{section}
\newtheorem{theorem}{Theorem}[section]
\newtheorem{lemma}[theorem]{Lemma}
\newtheorem{proposition}[theorem]{Proposition}

\theoremstyle{definition}

\theoremstyle{remark}
\newtheorem{remark}[theorem]{Remark}
\newtheorem{example}[theorem]{Example}

\newtheorem{acknowledgement}{Acknowledgement}

\newcommand{\Ass}{\operatorname{Ass}}

\newcommand{\Dim}{\operatorname{dim}}

\newcommand{\fm}{\frak{m}}
\newcommand{\fp}{\frak{p}}

\newcommand{\fa}{\frak{a}}
\newcommand{\fb}{\frak{b}}

\begin{document}

\dedicatory{Dedicated to Professor Winfried Bruns
on the occasion of his 70th birthday}
\author[Mafi]{Amir Mafi}
\title[Ratliff-Rush ideal and reduction numbers]
{Ratliff-Rush ideal and reduction numbers}

\address{A. Mafi, Department of Mathematics, University of Kurdistan, P.O. Box: 416, Sanandaj,
Iran.} \email{a\_mafi@ipm.ir}


\keywords{Reduction Number, Ratliff-Rush ideal.\\
This work was partially supported by a grant from the
Simons Foundation}

\begin{abstract}
Let $(R,\fm)$ be a Cohen-Macaulay local ring of positive dimension $d$ and infinite residue field. Let $I$ be an $\fm$-primary ideal and $J$ a minimal reduction of $I$. In this paper, we show that $\widetilde{r_J(I)}\leq r_J(I)$. This answer to a question that made by M.E. Rossi and I. Swanson in [\ref{Rs}, Question 4.6].
\end{abstract}

\maketitle

\section{Introduction}
Throughout this paper, we assume that $(R,\fm)$ is a Cohen-Macaulay (abbreviated to CM) local ring of positive dimension $d$ with infinite residue field and $I$ an $\fm$-primary ideal.

The ideals of the form $I^{m+1}:I^m=\{x\in R\vert\ \  xI^m\subseteq I^{m+1}\}$ increase with $m$. The union of this family first studied by Ratliff and Rush [\ref{Rr}]. Let us denote $\widetilde{I}=\cup_{m\geq 1}(I^{m+1}:I^m)$. The ideal $\widetilde{I}$ is called {\it the Ratliff-Rush ideal associated} with $I$ or {\it the Ratliff-Rush closure} of $I$. Ratliff and Rush showed that $\widetilde{I}$ is the largest ideal for which ${(\widetilde{I})}^m=I^m$ for all large $m$ and hence that $\widetilde{\widetilde{I}}=\widetilde{I}$. More generally, they proved that $\widetilde{I^m}=\cup_{k\geq 1}(I^{m+k}:I^k)$ and $\widetilde{I^m}=I^m$ for all large $m$; in particular, it holds that $$\widetilde{I^m}=\cup_{k\geq 1}(I^{m+k}:(x_1^k,...,x_d^k)),$$
where $x_1,...,x_d$ is a system of parameters contained in $I$. The ideal $I$ for which $\widetilde{I}=I$ is called {\it Ratliff-Rush closed}. There exist many ideals which are Ratliff-Rush ideals, for example, all radical and all integrally closed ideals. In [\ref{Hls}], Heinzer, Lantz and Shah showed that the depth of the associated graded ring $gr_I(R)=\bigoplus_{m\geq 0}\dfrac{I^m}{I^{m+1}}$ is positive iff all powers of $I$ are Ratliff-Rush ideals. For example, all power of an ideal is Ratliff-Rush, whenever it is generated by a regular sequence.

Recall that an ideal $J\subseteq I$ is called a reduction of $I$ if $I^{m+1}=JI^{m}$ for some non-negative integer $m$. A reduction $J$ is called a minimal reduction if $J$ is minimal with respect to inclusion and under our assumption, it is generated by a regular sequence. These concepts were first introduced and studied by Northcott and Rees [\ref{Nr}]. If $J$ is a reduction of $I$, define {\it the reduction number} of $I$ with respect to $J$, denoted by $r_J(I)$, to be $\min\{m \vert\ \  I^{m+1}=JI^m\}$. The reduction number of $I$ is defined by $r(I)=\min\{r_J(I)\vert\ \ J$ is a minimal reduction of $I\}$. The notion of minimal reduction can be given for filtrations and the extension is clear in the case of the Ratliff-Rush filtration. Since $\widetilde{I^m}=I^m$ for large $m$, a minimal reduction $J$ of $I$ is a minimal reduction with respect to the Ratliff-Rush filtration. Rossi and Swanson [\ref{Rs}] denoted by $$\widetilde{r_J(I)}=\min\{m\vert\ \ \widetilde{I^{n+1}}=J\widetilde{I^n}\ \ for\ n\geq m\},$$
and they called it the {\it Ratliff-Rush reduction number} of $I$ with respect to $J$. It is not clear whether $\widetilde{I^{m+1}}=J\widetilde{I^m}$ for some integer $m$ implies that $\widetilde{I^{n+1}}=J\widetilde{I^n}$ for all $n\geq m$. We remark in fact $\widetilde{I}\widetilde{I^m}$ is not necessarily $\widetilde{I^{m+1}}$.

Recall that an element $x$ of the ideal $I$ is said to be {\it superficial element} for $I$ if there exists a non-negative integer $k$ such that $(I^{m+1}:x)\cap I^k=I^m$ for all $m\geq k$ and so, with our assumption, there exists a non-negative integer $k_0$ such that  $(I^{m+1}:x)=I^m$ for all $m\geq k_0$. A set of elements $x_1,...,x_s\in I$ is a {\it superficial sequence} of $I$ if $x_i$ is a superficial element of $I/{(x_i,...,x_{i-1})}$ for $i=1,...,s$.
Swanson [\ref{S}] proved that if $x_1,...,x_d$ is a superficial sequence of $I$, then $J=(x_1,...,x_d)$ is a minimal reduction of $I$.
Elias [\ref{E}] defined that a superficial sequence $x_1,...,x_s$ of $I$ is {\it tame} if $x_i$ is a superficial element of $I$, for all $i=1,...,s$. Also, he proved that a tame superficial sequence always exists.

The main aim of this paper is to prove the
question that made by Rossi and Swanson in [\ref{Rs}, Question 4.6]. For any unexplained notation or terminology, we
refer the reader to [\ref{Bh}] and [\ref{Hs}].

\section{The results}
We shall need some auxiliary results. The following result was essentially proven in [\ref{K}, Lemma 3] and [\ref{C}, Theorem 2].

\begin{lemma} Let $x,x_1,...,x_s\in\fm$ be an $R$-regular sequence and $\fa=(x_1,...,x_s)$. Then the following conditions hold.
\begin{itemize}
\item[(i)]
${\fa}^{n+1}:x_i={\fa}^n$ for all $n\in\mathbb{N}$ and all $i$ $(1\leq i\leq s)$.
\item[(ii)] ${\fa}^{n}:x={\fa}^n$ for all $n\in\mathbb{N}$.
\end{itemize}
\end{lemma}

The next result is known. For the proof see [\ref{Mc}, Ch.VIII] and [\ref{P}, Theorem 2.2 ].

\begin{lemma} Let $m$ a non-negative integer. Then the following conditions hold.
\begin{itemize}
\item[(i)] $\widetilde{I^{m+1}}:x=\widetilde{I^m}$ for every superficial element $x\in I$.
\item[(ii)] $\widetilde{I^{m+1}}:J=\widetilde{I^m}$ for all minimal reduction $J$ of $I$.
\item[(iii)] $\widetilde{I^{m+1}}:I=\widetilde{I^m}$.
\item[(iv)] $J\widetilde{I^{m+1}}:x=\widetilde{I^{m+1}}$ for all minimal reduction $J$ of $I$ and all superficial element $x\in J$.
\end{itemize}
\end{lemma}

\begin{lemma} Let $(R,\fm)$ be a CM local ring of dimension two and let $x_1, x_2$ be a superficial sequence on I with $J=(x_1,x_2)$. Then
$J^{n+1}\widetilde{I^{m}}:x_1=J^n\widetilde{I^m}$ for all $m,n\in\mathbb{N}$.
\end{lemma}

\begin{proof}
We will prove the claim by induction on $n$. The case $n=0$  follows by Lemma 2.2. Assume that $n\geq 1$ and clearly $J^n\widetilde{I^m}\subseteq J^{n+1}\widetilde{I^{m}}:x_1$. Suppose $yx_1\in J^{n+1}\widetilde{I^{m}}$; then we have $yx_1=\alpha_1x_1+\alpha_2x_2$ for some $\alpha_1,\alpha_2\in J^{n}\widetilde{I^{m}}$. Thus $x_1(y-\alpha_1)=x_2\alpha_2\in x_2J^{n}\widetilde{I^{m}}$ and since $x_1, x_2$ is a regular sequence, we obtain $y-\alpha_1=tx_2$ for some $t\in R$. Since $x_1(y-\alpha_1)=tx_1x_2\in x_2J^{n}\widetilde{I^{m}}$ and $x_2$ is a non-zerodivisor, it follows that $tx_1\in J^{n}\widetilde{I^{m}}$ and so $t\in J^{n}\widetilde{I^{m}}:x_1$. Therefore, by induction hypothesis, $t\in J^{n-1}\widetilde{I^{m}}$ and hence $y\in J^n\widetilde{I^{m}}$, as requird.
\end{proof}
The following result was proved by Rossi and Swanson [\ref{Rs}]. We reprove with a simplified proof.
\begin{proposition} Let $(R,\fm)$ be a CM local ring of dimension two and let $x_1, x_2$ be a superficial sequence on I with $J=(x_1,x_2)$. If $r_J(I)=m$, then $\widetilde{r_J(I)}\leq m$.
 \end{proposition}

\begin{proof} Let $n$ be an integer such that $n\geq m$. Since $\widetilde{I^{n+1}}=I^{n+k+1}:(x_1^k,x_2^k)$ for all large $k$, we have $\widetilde{I^{n+1}}=J^{k+1}I^n:(x_1^k,x_2^k)\subseteq J^{k+1}\widetilde{I^{n}}:x_1^k$. By using Lemma 2.3, we have $\widetilde{I^{n+1}}\subseteq J\widetilde{I^{n}}$ and clearly $J\widetilde{I^{n}}\subseteq\widetilde{I^{n+1}}$. It therefore follows that $\widetilde{I^{n+1}}=J\widetilde{I^{n}}$ for all $n\geq m$ and so $\widetilde{r_J(I)}\leq m$.
\end{proof}

\begin{remark} Let $(R,\fm)$ be a CM local ring of dimension two and $J$ a minimal reduction of $I$. If $r_J(I)=m$ and $\widetilde{I^m}=I^m$, then by Proposition 2.4, we have $\widetilde{I^n}=I^n$ for all $n\geq m$.
\end{remark}

Let $G=\bigoplus_{m\geq 0}G_m$ be a Notherian graded ring where $G_0$ is an Artinian local ring, $G$ is generated by $G_1$ over $G_0$ and $G_{+}=\bigoplus_{m>0}G_m$. Let $H_{G_{+}}^i(G)$ denote the i-th local cohomology module of $G$ with respect to the graded ideal $G_+$ and set $a_i(G)=\max\{m\vert\ \  H_{G_{+}}^i(G)_m\neq 0\}$ with the convention $a_i(G)=-\infty$, if $H_{G_{+}}^i(G)=0$. The Castelnuovo-Mumford regularity is defined by $reg (G):=\max\{a_i(G)+i\vert\ \ i\geq 0\}$.

The following result can be also concluded by [\ref{Drt}, Theorem 2.4]
\begin{proposition}
Let $(R,\fm)$ be a CM local ring of dimension two and $J$ a minimal reduction of $I$ such that $r_J(I)=m$. If $\widetilde{I^m}=I^m$, then $r_J(I)=reg(gr_I(R))$.
\end{proposition}

\begin{proof} By using Remark 2.5, we have $\widetilde{I^n}=I^n$ for all $n\geq m$ and so for all superficial element $x\in I$, we have $(x)\cap I^{n+1}=xI^n$ for all $n\geq m$. Therefore by [\ref{T1}, Proposition 4.7], $reg(gr_I(R))\leq m$ and also by [\ref{T}, Proposition 3.2] we have $r_J(I)\leq reg(gr_I(R))$ (see also [\ref{M}, Lemma 1.2]). Hence $r_J(I)=reg(gr_I(R))$
\end{proof}

Let $(R,\fm)$ be a Noetherian local ring with $\Dim R>0$ and $I$ an ideal of $R$. A system of homogeneous elements $y^{*}_1,...,y^{*}_t$ in $gr_{I}(R)$ is called filter-regular sequence if and only if $$y^{*}_i\notin\bigcup_{\fp\in\Ass(gr_{I}(R)/{(y^{*}_1,...,y^{*}_{i-1})gr_{I}(R)})\setminus V(gr_{I}(R)_{+})}\fp$$ for $i=1,...,t$ (see [\ref{T}]). Trung in [\ref{T1}, Lemma 6.2] proved that $y_1,...,y_t$ is a superficial sequence of $I$ if and only if $y_1^{*},...,y_t^{*}$ forms a filter-regular sequence of $gr_{I}(R)$, where $y_i^{*}=y_i+{I}^2$.
A filter-regular $y^{*}_1,...,y^{*}_t$ in any order is called unconditioned filter-regular sequence.

\begin{lemma} Let $(R,\fm)$ be a CM local ring and $I$ an $\fm$-primary ideal of $R$. Then every minimal reduction $J$ of $I$ can be generated by a tame superficial sequence of $I$.
\end{lemma}

\begin{proof}
Let $J$ be a minimal reduction of $I$. Since the residue field is infinite, it is well known that there exists a superficial sequence $y_1,..,y_d$ in $I$ such that $J=(y_1,..,y_d)$, see for instance section 8.6 in [\ref{Hs}]. Since $y^{*}_1,...,y^{*}_d$ forms a filter-regular sequence in $gr_I(R)$, then by
[\ref{Tz}, Proposition 1.2] $Q=(y^{*}_1,...,y^{*}_d)$ admits a system of generators of length $d$, say $z^{*}_1,...,z^{*}_d$, which forms an unconditional $Q$-filter-regular sequence. Since $z_1,...,z_d$ is in particular a maximal superficial sequence in the minimal reduction $J$, it follows that $J=(z_1,...,z_d)$ is generated by a tame superficial sequence.
\end{proof}

\begin{lemma} Let $(R,\fm)$ be a CM local ring of dimension $d\geq 3$, $I$ an $\fm$-primary ideal and $x,x_1,...,x_s$ a tame superficial sequence of $I$. If $\fa=(x_1,...,x_s)$, then ${\fa}^n\widetilde{I^{m+1}}:x={\fa}^n\widetilde{I^m}$ for all non-negative integers $m,n$.
\end{lemma}

\begin{proof} We proceed by induction on $n$. The case $n=0$ is trivial. Now we assume that $n\geq 1$ and by induction hypothesis ${\fa}^n\widetilde{I^{m+1}}:x\subseteq{\fa}^{n-1}\widetilde{I^{m+2}}:x={\fa}^{n-1}\widetilde{I^{m+1}}$. So that the argument is finished once we prove the following:\\
$(x_1,...,x_r)^{n-1}\widetilde{I^{m+1}}\cap({\fa}^n\widetilde{I^{m+1}}:x)\subseteq {\fa}^n\widetilde{I^m}$ for all $r$ $(0\leq r\leq s)$.
Again, we proceed by induction on $r$. For $r=0$, we take $(x_1,...,x_r)=0$. Assume $r\geq 1$ and $y$ an element of $(x_1,...,x_r)^{n-1}\widetilde{I^{m+1}}\cap({\fa}^n\widetilde{I^{m+1}}:x)$. We can put $y=\alpha+x_r\beta$, where $\alpha\in(x_1,...,x_{r-1})^{n-1}\widetilde{I^{m+1}}$  and $\beta\in (x_1,...,x_r)^{n-2}\widetilde{I^{m+1}}$. Now let $\fb=(x_1,...,\widehat{x_r},...,x_s)$ and then $xy=x\alpha+xx_r\beta\in{\fa}^n\widetilde{I^{m+1}}={\fb}^n\widetilde{I^{m+1}}+x_r{\fa}^{n-1}\widetilde{I^{m+1}}$. Thus we can find an element $z$ of ${\fa}^{n-1}\widetilde{I^{m+1}}$ such that $xy-x_rz=x\alpha+x_r(x\beta-z)\in{\fb}^n\widetilde{I^{m+1}}$. Since ${\fb}^n\widetilde{I^{m+1}}\subseteq{\fb}^{n-1}\widetilde{I^{m+2}}$ and $\alpha\in{\fb}^{n-1}\widetilde{I^{m+1}}$, we get $x\alpha\in{\fb}^{n-1}\widetilde{I^{m+2}}$. Hence $x_r(x\beta-z)\in{\fb}^{n-1}\widetilde{I^{m+2}}$ and by induction hypothesis on $n$, we have $x\beta-z\in{\fb}^{n-1}\widetilde{I^{m+1}}$ and so $x\beta\in{\fa}^{n-1}\widetilde{I^{m+1}}$ (as ${\fb}^{n-1}\widetilde{I^{m+1}}\subseteq{\fa}^{n-1}\widetilde{I^{m+1}}$). Again, by induction hypothesis on $n$, we have $\beta\in{\fa}^{n-1}\widetilde{I^{m}}$. Therefore $x\alpha=xy-xx_r\beta\in{\fa}^n\widetilde{I^{m+1}}$ and so $\alpha\in({\fa}^n\widetilde{I^{m+1}}:x)\cap(x_1,...,x_{r-1})^{n-1}\widetilde{I^{m+1}}$. Thus by induction hypothesis on $r$, we have $\alpha\in{\fa}^n\widetilde{I^{m}}$ so that $y=\alpha+x_r\beta$ is contained in ${\fa}^n\widetilde{I^{m}}$, as desired.
\end{proof}

\begin{lemma} Let $(R,\fm)$ be a CM local ring of dimension $d\geq 3$, $I$ an $\fm$-primary ideal, and $x_1,x_2,...,x_d$ a tame superficial sequence of $I$.
If $J=(x_1,x_2,...,x_d)$, then $J^{n+1}\widetilde{I^{m}}:x_1=J^n\widetilde{I^{m}}$ for all non-negative integers $m,n$.
\end{lemma}

\begin{proof}
Let us proceed by induction on $n$. The case $n=0$ follows by Lemma 2.2. Let $n\geq 1$. We have $x_1(J^{n+1}\widetilde{I^{m}}:x_1)=J^{n+1}\widetilde{I^{m}}\cap(x_1)=(J_1^{n+1}\widetilde{I^{m}}+x_1J^n\widetilde{I^{m}})\cap(x_1)=
J_1^{n+1}\widetilde{I^{m}}\cap(x_1)+x_1J^n\widetilde{I^{m}}=x_1(J_1^{n+1}\widetilde{I^{m}}:x_1)+x_1J^n\widetilde{I^{m}}$, where $J_1=(x_2,...,x_d)$.
Therefore by using Lemma 2.8, we have $x_1(J^{n+1}\widetilde{I^{m}}:x_1)=x_1J_1^{n+1}\widetilde{I^{m-1}}+x_1J^n\widetilde{I^{m}}=x_1J^n\widetilde{I^{m}}$ (as $J_1^{n+1}\widetilde{I^{m-1}}\subseteq J^n\widetilde{I^{m}}$). Hence $J^{n+1}\widetilde{I^{m}}:x_1=J^n\widetilde{I^{m}}$, as desired.

\end{proof}

\begin{theorem} Let $(R,\fm)$ be a CM local ring of dimension $d\geq 3$, $I$ an $\fm$-primary ideal $x_1,x_2,...,x_d$ a tame superficial sequence of $I$.
If $J=(x_1,x_2,...,x_d)$, then $\widetilde{r_J(I)}\leq r_J(I)$.

\end{theorem}

\begin{proof} Let us write $r_J(I)=m$ and we prove $\widetilde{I^{m+1}}=J\widetilde{I^m}$. For large $k$, we have $\widetilde{I^{m+1}}=I^{m+k+1}:(x_1^k,x_2^k,...,x_d^k)$, in particular $\widetilde{I^{m+1}}=J^{k+1}I^m:(x_1^k,x_2^k,...,x_d^k)$ as $I^{m+n}=J^nI^m$ for all non-negative integers $n$.
Therefore by using Lemma 2.9, we have $\widetilde{I^{m+1}}=J^{k+1}I^m:(x_1^k,...x_d^k)\subseteq J^{k+1}\widetilde{I^{m}}:x_1^k=J\widetilde{I^{m}}$. Hence $\widetilde{I^{m+1}}\subseteq J\widetilde{I^{m}}$ and so $\widetilde{I^{m+1}}=J\widetilde{I^{m}}$. This complete the proof.
\end{proof}

\begin{remark} Let $(R,\fm)$ be a CM local ring of dimension $d\geq 3$, $x_1,x_2,...,x_d$ a tame superficial sequence of $I$ and $J=(x_1,x_2,...,x_d)$. If $r_J(I)=m$ and $\widetilde{I^m}=I^m$, then by Theorem 2.10, we have $\widetilde{I^n}=I^n$ for all $n\geq m$.
\end{remark}

The following example shows that $I^m$ is not necessarily Ratliff-Rush closed even if $m\geq r_J(I)$.
The computations are performed by using Maxaulay2 [\ref{Gs}].

\begin{example} Let $R=k[\![ x,y ]\!]$, where $k$ be a field and $I=(x^7,x^6y,x^2y^5,y^7)$. Then $r(I)=3$ and $x^{17}y^4\in I^4:I\setminus I^3$. Hence $\widetilde{I^3}\neq I^3$.

\end{example}

The following example show that the inequality in Theorem 2.10 can be strict.

\begin{example}
Let $R=k[\![ x,y ]\!]$, where $k$ be a field and $I=(x^4,x^3y,xy^3,y^4)$. Then $r(I)=2$, $e_2(I)=0$ and so by Huckaba and Marley[\ref{Hm}, Corollary 4.13] $\widetilde{r_J(I)}\leq 1$ for all minimal reduction $J$ of $I$, where $e_2(I)$ is the second Hilbert coefficients.
\end{example}

\begin{acknowledgement} This paper was done while I was visiting the University of Osnabruck. I would like to thank the Institute of Mathematics of the University of Osnabruck for hospitality. Also, I would like to express my deep gratitude to Professor Louis Ratliff and Professor Tony Puthenpurakal for valuable suggestions. Finally, I would like grateful to the referee for the careful reading of the manuscript and the helpful suggestions.
\end{acknowledgement}

\end{document}